\documentclass[12pt]{amsart}

\usepackage{amsmath,amssymb,amscd,amsxtra}
\usepackage{latexsym}
\usepackage[dvips]{graphics,epsfig}
\usepackage{enumerate}
\usepackage[colorlinks=true, pdfstartview=FitV, linkcolor=blue, citecolor=blue, urlcolor=blue]{hyperref}
\newtheorem{theorem}{Theorem}
\newtheorem{corollary}[theorem]{Corollary}

\theoremstyle{definition}

\theoremstyle{remark}
\newtheorem{remark}{Remark}[section]

\numberwithin{equation}{section}

\newcommand{\al}{\alpha}
\newcommand{\be}{\beta}
\newcommand{\ga}{\gamma}

\newcommand{\bn}{{\mathbb N}}
\newcommand{\na}{{\mathbb N}}

\newcommand{\re}{{\mathbb R}}

\newcommand{\rtn}{{\mathbb R^{2n}}}

\newcommand{\rn}{{\mathbb R^n}}
\newcommand{\brn}{{\mathbb R^n}}

\newcommand{\zp}{\mathbf Z^+}
\newcommand{\zptn}{\mathbf Z^{2n}_+}
\newcommand{\zpn}{\mathbf Z^{n}_+}

\newcommand{\norm}[2]{\left\|#1\right\|_{#2}}

\newcommand{\supp}{\operatorname{supp}}
\newcommand{\abs}[1]{\left|#1\right|}

\newcommand{\bq}{\begin{equation}}
\newcommand{\eq}{\end{equation}}

\begin{document}

\title[On the H\"ormander classes of bilinear pseudodifferential operators]{On the H\"ormander classes of bilinear pseudodifferential operators}

\author{\'Arp\'ad B\'enyi}
\author{Diego Maldonado}
\author{Virginia Naibo}
\author{Rodolfo H.  Torres}




\thanks{Second author partially supported by the NSF under grant DMS 0901587. Fourth author supported in part by NSF  under grant DMS 0800492 and a General Research Fund allocation of the University of Kansas.}

\address{\'Arp\'ad B\'enyi, Department of Mathematics,
516 High St, Western Washington University, Bellingham, WA 98225,
USA.} \email{arpad.benyi@wwu.edu}

\address{Diego Maldonado, Department of Mathematics, 138 Cardwell Hall, Kansas State University,
Manhattan, KS 66506, USA.} \email{dmaldona@math.ksu.edu}

\address{Virginia Naibo, Department of Mathematics, 138 Cardwell Hall, Kansas State University,
Manhattan, KS 66506, USA.} \email{vnaibo@math.ksu.edu}

\address{Rodolfo Torres, Department of Mathematics, University of Kansas, Lawrence, KS
66045, USA.} \email{torres@math.ku.edu}

\date{\today}

\begin{abstract}
Bilinear
pseudodifferential operators with symbols in the bilinear analog of all the
H\"ormander classes are considered and the possibility of a symbolic calculus for the transposes of the operators in such classes is investigated. Precise results about which classes
are closed under transposition and can be characterized in terms
of asymptotic expansions are presented. This  work extends the results for more limited classes studied before in the literature and, hence,  allows the use of the symbolic calculus (when it exists) as an alternative way to recover the  boundedness on products of Lebesgue spaces for  the classes that yield operators with bilinear Calder\'on-Zygmund kernels.
Some boundedness properties for other classes  with estimates in the form of Leibniz' rule are presented as well.
\end{abstract}

\maketitle

\section{Introduction}

Many linear operators encountered in analysis are best understood when represented as  singular integral operators in the space domain, while others are better treated as pseudo-differential operators in the frequency domain.  In some particular situations both representations are readily available. In many others, however, one of them is only given abstractly, and  through the existence of a distributional  kernel or symbol which is hard or impossible to compute.  Both representations have proved to be tremendously useful. The representation of operators as pseudodifferential ones usually yields simple $L^2$ estimates, explicit formulas for the calculus of transposes and composition, and invariance properties under change of coordinates in smooth situations.
As it is known, this makes pseudodifferential operators an invaluable tool in the study of partial differential equations and they are employed to construct parametrices and study regularity properties of solutions.
The integral representation on the other hand, is often best suited for other $L^p$ estimates and motivates  or indicates what results should hold in other metric and measure theoretic situations where the Fourier transform is no longer available. This has
found numerous applications in complex analysis, operator theory,  and also in problems in partial differential equations where the  domains or functions involved have a  minimum amount of regularity.

Work on singular integral and pseudodifferential operators started with explicit classical examples and  was then directed to attack specific applications in other areas. Switching back and forth, many efforts where also oriented to the understanding of  naturally appearing technical questions and to the testing of the full power of new techniques as they developed. In fact, sometimes the technical analytic tools studied preceded  the applications in which they were much later used. The Calder\'on-Zygmund theory and the related real variable techniques played a tremendous role in all these accomplishments. This is explained in detail from both the historical  and technical points of view in, for example,  the book by Stein \cite{S}.

The study of bilinear operators within harmonic analysis is following a similar path. The first systematic treatment of bilinear singular integrals and pseudodifferential  operators  in the early work of
Coifman and Meyer \cite{cm0}, \cite{cm2}  originated from specific problems about Calder\'on's commutators, and soon lead to the study of general boundedness properties of pseudodifferential operators \cite{cm1}.  Later on, the work  of Lacey and Thiele \cite{lt1}, \cite{lt2}  on a specific singular integral operator (the bilinear Hilbert transform which also goes back to Calder\'on) and the new techniques developed by them immediately suggested the study of other bilinear operators and the need to understand and characterize their boundedness and computational properties. See for example Gilbert and Nahmod \cite{GN},  Muscalu et al. \cite{mtt}, Grafakos and Li \cite{GL2}, B\'enyi et al. \cite{BDNTTV}, to name a few.  The development of the symbolic calculus for bilinear pseudodifferential operators started in  B\'enyi and Torres \cite{bt1}  and was continued in B\'enyi et al. \cite{bnt}. Other  results specific to bilinear pseudodifferential operators
were obtained in B\'enyi et al. \cite{bt2},  \cite{b}, \cite{bgho},  \cite{bo}  and, much recently,  in Bernicot \cite{Ber1} and  Bernicot and Torres \cite{BerT}.
As in the linear case, many of the results obtained were motivated too by the
Calder\'on-Zygmund theory and its bilinear counterpart as developed in Grafakos and Torres \cite{gt1}; see also  Christ and Journ\'e \cite{CJ}, Kenig and Stein \cite{KS}, Maldonado and Naibo \cite{mn}.   The literature is by  now vast, see
\cite{T} for further references.

We want to contribute with this article to the understanding of the properties of  {\it all} the bilinear analogs of  the linear H\"ormander classes of pseudodifferential operators. These bilinear classes are denoted by $BS^m_{\rho, \delta}$ (see the next section for technical definitions). Only some particular cases of them have been studied before;
mainly the cases when $\rho =1$ or when $\rho=\delta=0$.  The symbolic calculus has only been developed for the case $\rho=1$ and $\delta=0$. Our goal is to complete the symbolic calculus for all the possible (and meaningful) values of
$\delta$ and $\rho$.

 We could quote from the introduction of H\"ormander's work \cite{H}: ``In this work the use of Fourier transformations has been emphasized; as a result no singular integral operators are apparent...", but
 we are clearly guided by previous works that relate, in the case of operators of order zero, to bilinear
 Calder\'on-Zygmund singular integrals. In fact, the existence of calculus for $BS^0_{1,\delta}$, $\delta<1$, gives an
 alternative way to prove the boundedness of such operators in the optimal range of  $L^p$ spaces directly from the multilinear $T1$-Theorem in \cite{gt1}. That is, without using Littlewood-Paley arguments as in the already cited monograph  \cite{cm1}
 or the work \cite{b}.

 While the composition of pseudodifferential operators (with linear ones)
 forces one to study  different classes of operators introduced in \cite{bnt},  previous results in the subject left some level of uncertainty about whether the computation of transposes
 could still  be accomplished within some other  bilinear H\"ormander classes.
The forerunner work \cite{bt1} dealt mainly with the class $BS^0_{1,0}$ and a significant
part of the proofs given in \cite{bt1} relied on the
so-called Peetre inequality which does not go through in general for other values of  $\rho\neq 1, \delta\neq0, m\neq0$.
  We resolve this problem in the present article  using  ideas inspired in part by some computations in Kumano-go \cite{Kumano-go}, and developing  the calculus of transposes for all the bilinear H\"ormander classes for which such calculus is possible.  The excluded classes are the ones for which $\rho=\delta=1$. This restriction is really necessary as proved in \cite{bt1}. In fact, as the linear class $S^0_{1,1}$, the class $BS^0_{1,1}$ is forbidden in the sense that two related pathologies occur: it is not closed under transposition and it contains
operators which fail to be bounded on product of $L^p$ spaces, even though  the associated kernels
for operators in this class are of bilinear Calder\'on-Zygmund type.

The analogy between results in the linear and multilinear situations is in general
only a guide to what could be expected to transfer from one context to the other.
Some multilinear results arise as natural counterparts to linear ones, but often the
techniques employed need to be substantially sharpened or replaced by new ones. It
is actually far more complicated to prove the existence of calculus in the bilinear case
than in the linear one. Some properties of the symbols of bilinear pseudodifferential operators on $\rn$ can be guessed
from those of linear operators in $\rtn$. Though some of our computations are reminiscent of those for
linear pseudodifferential operators or Fourier integral operators, the calculus of transposes
for bilinear operators does not follow from the linear results by doubling the number of dimensions.
Boundedness results cannot be obtained in this fashion either. The essential obstruction is the fact that the integral of a function of two $n$-dimensional variables $(x,y) \in \re^{2n}$ yields no information about the ($n$-dimensional) integral of its restriction to the diagonal $(x,x), x \in \rn$. On the other hand, a few \emph{point-wise }estimates can be obtained in a
more direct way from the linear case using the method of doubling the dimensions. For example,
it is useful to establish first precise point-wise estimates on the bilinear kernels associated to the operators in various
H\"ormander classes. We are able to derive them from the linear ones investigated by \'Alvarez-Hounie \cite{AH90}.

It is interesting too that some  results do not extend to the
multilinear context. A notorious example is the Calder\'on-Vaillancourt result in \cite{CV} for  the $L^2$ boundedness of the class $S^0_{0,0}$. One may expect the class $BS^0_{0,0}$ to map, say, $L^2 \times L^2 $ into $L^1$, but this fails  unless additional conditions on the symbol are imposed; see \cite{bt2}. Such class only maps into an optimal modulation space  ($M^{1,\infty}$) which is larger than $L^1$. See also \cite{bgho} and \cite{bo} for more details. Similarly (and using duality and the existence of the calculus for transposes), it is natural to ask whether  the class $BS^0_{\rho,\delta},$ with $0<\delta<\rho < 1$
maps $L^2 \times L^\infty$ into $L^2$ -- recall that H\"ormander's results in  \cite{H} give  that $S^0_{\rho,\delta}$
maps $L^2$ into $L^2$ for the same range of $\rho$ and $\delta$. Alternatively, it may only be possible to obtain the boundedness from $L^2 \times X$ into $L^2,$ where $X$ is a  space smaller than $L^\infty$.  Though we do not know the answer to the former questions, we give a result in the direction of the latter. We also obtain some other new boundedness properties involving Sobolev spaces which take the form of fractional Leibniz' rules.

In the next section we present the technical definitions and the precise statements of our results
about symbolic calculus. Sections 3 and 4  contain the proofs of those results. Section 5 contains the results about the point-wise estimates for the kernels, while Section 6 has the boundedness results alluded to before.\\

{\bf Acknowledgement.}  The authors would like to thank Kasso Okoudjou for  useful discussions regarding the results presented here and to the  anonymous
referee for his/her comments and suggestions.

\section{Symbolic  calculus in the bilinear  H\"ormander classes}

We start by recalling the linear pseudodifferential operators in the general H\"ormander classes $S^m_{\rho,\delta}$. These are operators of the form
\[
T_\sigma (f)(x)=\int_\brn \sigma (x, \xi)\widehat
f(\xi)e^{ix\cdot \xi}\, d\xi
\]
where the symbol $\sigma$ satisfies the estimates
\begin{equation*}
|\partial_x^\alpha\partial_\xi^\beta \sigma (x,
\xi)|\leq C_{\alpha\beta} (1+|\xi|)^{m+\delta
|\alpha|-\rho |\beta|},
\end{equation*}
for all $x,\,\xi \in\brn$, all multi-indices $\alpha,
\beta,$ and some positive constants
$C_{\alpha\beta}$. These operators are a priori defined for appropriate test functions.

In this article we study the natural bilinear analog
\[
T_\sigma (f, g)(x)=\int_\brn \int_\brn \sigma (x, \xi, \eta)\widehat
f(\xi)\widehat g(\eta)e^{ix\cdot (\xi+\eta)}\, d\xi d\eta,
\]
where the symbol $\sigma$ satisfies now the estimates
\begin{equation}\label{hormander}
|\partial_x^\alpha\partial_\xi^\beta\partial_\eta^\gamma \sigma (x,
\xi, \eta)|\leq C_{\alpha\beta\gamma} (1+|\xi|+|\eta|)^{m+\delta
|\alpha|-\rho (|\beta|+|\gamma|)},
\end{equation}
also for all $x,\,\xi,\,\eta\in\brn$, all multi-indices $\alpha,
\beta, \gamma$ and some positive constants
$C_{\alpha\beta\gamma}.$ The class of all symbols satisfying
\eqref{hormander} is denoted by $BS_{\rho, \delta}^m (\brn),$ or simply
$BS_{\rho, \delta}^m$ when it is clear from the context to which space the variables
$x, \xi, \eta$ belong to.

The transposes of such operators are defined as usual by the duality
relations
\[
\langle T(f, g), h\rangle=\langle T^{*1}(h, g), f \rangle=\langle
T^{*2}(f, h), g\rangle.
\]

We  will write
$$\sigma \sim \sum_{j=0}^\infty \sigma_j$$
if there is a  non-increasing sequence $m_N\searrow -\infty$ such
that
$$\sigma-\sum_{j=0}^{N-1} \sigma_j\in BS_{\rho, \delta}^{m_N},$$
for all $N>0$.

The spaces of test functions that we will use  will be the space $C^\infty_c$ of  infinitely differentiable  functions with compact support  or the Schwartz space $\mathcal S$. When given their usual topologies, their duals
are $\mathcal D'$  and $\mathcal S'$, the spaces of distributions and  of tempered distributions, respectively. We will also consider  $C_c^s$, $s\in\na,$ the space of functions with compact support and continuous derivatives up to order $s$; $W^{s, 2}$,  the Sobolev space of functions having derivatives in $L^2$  up to order $s$; and $W_0^{s, \infty}$, the  completion of  $C_c^s$ with respect to the norm $\sup_{\abs{\gamma}\le s}\|D^{\gamma}g\|_{L^\infty}$. Unless specified otherwise, the underlying space will be assumed to be $\brn$.

We develop a symbolic calculus for bilinear
pseudodifferential operators with symbols in all the bilinear
H\"ormander classes $BS_{\rho, \delta}^m,$ $m\in\mathbb R,$ $0\leq
\delta\le\rho\leq 1,$ $\delta<1.$

Our first two theorems state that the H\"ormander classes are closed
under transposition and that the symbols of the transposed operators have
appropriate asymptotic expansions.

\begin{theorem}[Invariance under transposition]\label{symbolic}
Assume that $0\leq \delta\le\rho\leq 1,$ $\delta<1,$ and $\sigma\in
BS_{\rho, \delta}^m$. Then, for $j=1,\,2,$
$T_\sigma^{*j}=T_{\sigma^{*j}},$ where $\sigma^{*j}\in BS_{\rho,
\delta}^m$.
\end{theorem}

\begin{theorem}[Asymptotic expansion]\label{symbolic2}

If $0\leq \delta<\rho\leq 1$ and $\sigma\in BS_{\rho, \delta}^m,$
then $\sigma^{*1}$ and $\sigma^{*2}$ have asymptotic expansions
\begin{equation*}
\sigma^{*1}\sim\sum_{\al}
\frac{i^{\abs{\alpha}}}{\alpha!}\partial_x^\alpha\partial_\xi^\alpha
(\sigma(x,-\xi-\eta,\eta))
\end{equation*}
and
\begin{equation*}
\sigma^{*2}\sim\sum_{\al}
\frac{i^{\abs{\alpha}}}{\alpha!}\partial_x^\alpha\partial_\eta^\alpha
(\sigma(x,\xi,-\xi-\eta)).
\end{equation*}
More precisely, if $N\in \bn$ then
\begin{equation}\label{expansion1}
\sigma^{*1}-\sum_{\abs{\al}<N}\frac{i^{\abs{\alpha}}}{\alpha!}\partial_x^\alpha\partial_\xi^\alpha
(\sigma(x,-\xi-\eta,\eta))\in BS^{m+(\delta-\rho)N}_{\rho, \delta}
\end{equation}
and
\begin{equation}\label{expansion2}
\sigma^{*2}-\sum_{\abs{\al}<N}
\frac{i^{\abs{\alpha}}}{\alpha!}\partial_x^\alpha\partial_\eta^\alpha
(\sigma(x,\xi,-\xi-\eta))  \in BS^{m+(\delta-\rho)N}_{\rho, \delta}.
\end{equation}
\end{theorem}

In relation to asymptotic expansions we also prove the following two theorems.

\begin{theorem}\label{symbolic3}
Assume that $a_j\in BS_{\rho, \delta}^{m_j}, j\geq 0$ and
$m_j\searrow -\infty$ as $j\rightarrow\infty$. Then, there exists
$a\in BS_{\rho, \delta}^{m_0}$ such that $a\sim
\displaystyle\sum_{j=0}^\infty a_j$. Moreover, if
$$b\in BS_{\rho, \delta}^\infty=\mathop{\cup}_{m} BS_{\rho, \delta}^m\, \text{ and }\,
b\sim\displaystyle\sum_{j=0}^\infty a_j,$$ then
$$a-b\in BS_{\rho, \delta}^{-\infty}=\mathop{\cap}_{m} BS_{\rho, \delta}^m.$$
\end{theorem}

\begin{theorem}\label{symbolic4}
Assume that $a_j\in BS_{\rho, \delta}^{m_j}, j\geq 0$ and
$m_j\searrow -\infty$ as $j\rightarrow\infty$. Let $a\in C^\infty
(\brn\times\brn\times\brn)$ be such that \bq\label{egy}
|\partial_x^\al\partial_\xi^\be\partial_\eta^\ga a (x, \xi,
\eta)|\leq C_{\al\be\ga}(1+|\xi|+|\eta|)^{\mu}, \eq for some
positive constants $C_{\al\be\ga}$ and $\mu=\mu (\al, \be, \ga)$. If there
exist $\mu_N\rightarrow \infty$ such that \bq\label{ketto} |a(x,
\xi, \eta)-\sum_{j=0}^{N} a_j(x, \xi, \eta)|\leq C_N
(1+|\xi|+|\eta|)^{-\mu_N}, \eq then $a\in BS_{\rho, \delta}^{m_0}$
and $a\sim\displaystyle\sum_{j=0}^\infty a_j.$
\end{theorem}

The continuity on Lebesgue spaces of bilinear pseudodifferential
operators with symbols in the class $BS^0_{1,\delta}$ with $0\leq
\delta < 1$ has been intensely addressed in the literature.
 It is nowadays a well-known fact that the
bilinear kernels associated to bilinear operators with symbols in
$BS^0_{1,\delta}$, $0 \leq \delta < 1$, are bilinear
Calder\'on-Zygmund operators in the sense of Grafakos and Torres
\cite{gt1}.  Recall the
 following  result (\cite[Corollary~1]{gt1}) which is an application of the bilinear $T1$-Theorem therein:

{\it
If $T$ and its transposes, $T^{*1}$ and  $T ^{*2}$, have symbols in
$BS_{1, 1}^0$, then they can be extended as
bounded operators from $L^p\times L^q$
into $L^r$ for $1<p,q<\infty$ and $1/p+1/q=1/r$.}

As mentioned in the introduction, and since $BS^0_{1,\delta} \subset BS^0_{1,1}$,  we can directly combine this result with Theorem ~\ref{symbolic} to recover the following optimal version of a known fact.
\begin{corollary}\label{corollary2}
If $\sigma$ is a symbol
in $BS_{1,\delta}^0$, $0\leq\delta<1,$ then $T_\sigma$ has a bounded extension from
$L^p\times L^q$ into $L^r,$ for all $1<p,q<\infty, 1/p+1/q=1/r$.
\end{corollary}

\section{Proofs of Theorem \ref{symbolic} and Theorem \ref{symbolic2}}\label{secc:sym12}

In the following, we assume that the symbol $\sigma$ has compact
support (in all three variables $x$, $\xi$, and $\eta$) so that the
calculations in the proofs of Theorem \ref{symbolic} and Theorem~
\ref{symbolic2}  are properly justified. All estimates are obtained
with constants independent of the support of $\sigma$ and an
approximation argument can be used to obtain the results for symbols
that do not have compact support; see \cite{bt1} for further details regarding
such an approximation argument.

We restrict the proofs of Theorem \ref{symbolic} and Theorem~
\ref{symbolic2} to the first transpose of $T_\sigma$ ($j=1$).
As in \cite{bt1}, we rewrite
$T_\sigma^{*1}$ as a compound operator. We have
$$T^{*1}(h, g)(x)=\int_y\int_\eta\int_\xi c(y, \xi, \eta)h(y)\widehat g(\eta)e^{-i(y-x)\cdot\xi}e^{ix\cdot\eta}d\xi d\eta dy,$$
where
$$c(y, \xi, \eta)=\sigma (y, -\xi-\eta, \eta).$$

Straightforward calculations show that $c$ satisfies the same
differential inequalities as $\sigma$. Indeed, by the Leibniz rule
we can write
\begin{align}\begin{split}\label{cesti}
|\partial_y^\al\partial_\xi^\be\partial_\eta^\ga c(y, \xi,
\eta)|&\lesssim \sum_{|\ga_1|+|\ga_2|=|\ga|}
|\partial_y^\al\partial_\xi^\be\partial_{\eta2}^{\ga_1}\partial_{\eta3}^{\ga_2}\sigma (y, -\xi-\eta, \eta)|\\
&\lesssim \sum (1+|\xi+\eta|+|\eta|)^{m+\delta|\al|-\rho (|\be|+|\ga_1|+|\ga_2|)}\\
&\lesssim (1+|\xi|+|\eta|)^{m+\delta |\al|-\rho (|\be|+|\ga|)},
\end{split}
\end{align}
with constants of the form $\sum_{|\ga_1|+|\ga_2|=|\ga|}
C_{\al\be\ga_j} 2^{m+\delta|\al|+\rho (|\be|+|\ga|)}.$

By appropriately changing variables of integration the symbol of $T^{*1}$ is given in terms of $c$ by
the following expressin:
\begin{equation}\label{arepresentation}
a(x, \xi, \eta)=\int\int c(x+y, z+\xi, \eta)e^{-iz\cdot y}\, dydz.
\end{equation}

\bigskip

\begin{proof}[Proof of Theorem~\ref{symbolic}]

We will use the representation \eqref{arepresentation} of $a$ to
show that $a\in  BS_{\rho, \delta}^m.$ By \eqref{cesti} and  since
$\partial_x^\alpha\partial_\xi^\beta\partial_\eta^\gamma a(x, \xi,
\eta)=\int \int
\partial_x^\alpha\partial_\xi^\beta\partial_\eta^\gamma c(x+y,
z+\xi, \eta) e^{-iz\cdot y}\,dydz,$ it is enough to work with
$\alpha=\beta=\gamma=0.$ Our techniques are inspired in part by ideas in \cite[Lemma 2.4, page
69]{Kumano-go}.

In the following, fix $\xi\in\brn,$  $\eta\in\brn,$ and set
$A:=1+|{\xi}|+|{\eta}|.$ We have to prove that
\[
\abs{a(x,\xi,\eta)}\lesssim A^m
\]
with a constant independent of the support of $\sigma.$

 Let $l_0\in\bn,$ $2l_0>n.$ Writing
\begin{equation*}
e^{-iz\cdot
y}=(1+A^{2\delta}\abs{y}^2)^{-l_0}(1+A^{2\delta}(-\Delta_z))^{l_0}
e^{-iz\cdot y},
\end{equation*}
integration by parts gives
\begin{equation}\label{awithq}
a(x,\xi,\eta)=\int\int q(x,y,z,\xi,\eta)e^{-iz\cdot y}\,dydz,
\end{equation}
where
\[q(x,y,z,\xi,\eta)=\frac{(1+A^{2\delta}(-\Delta_z))^{l_0}c(x+y,z+\xi,\eta)}{(1+A^{2\delta}\abs{y}^2)^{l_0}}.\]

We now estimate $(-\Delta_y)^lq$ for $l\in\bn.$
\begin{align}
&(-\Delta_y)^lq=\mathop{\sum_{\abs{\alpha}=2l}}_{\alpha_i
\text{ even}} C_\alpha\,\partial_y^\alpha q(x, y, z, \xi, \eta) \label{lappower}\\
&= \mathop{\sum_{\abs{\alpha}=2l}}_{\alpha_i \text{
even}}\sum_{{\beta}\le \alpha}
C_{\alpha\beta}\,\partial_y^{\beta}\left((1+A^{2\delta}\abs{y}^2)^{-l_0}\right)\partial_y^{\alpha-\beta}\left(
(1+A^{2\delta}(-\Delta_z))^{l_0}c(x+y,z+\xi,\eta)\right).\nonumber
\end{align}
Note that
\begin{equation}\label{esti1}
\abs{\partial_y^{\beta}\left((1+A^{2\delta}\abs{y}^2)^{-l_0}\right)}\le
C_{\beta
l_0}\,A^{\delta\abs{\beta}}\left(1+A^{2\delta}\abs{y}^2\right)^{-l_0}.
\end{equation}
Moreover, if $P_{l_0}=\{\gamma=(\gamma_1,\cdots,\gamma_n): \gamma_i
\text{ even and } \abs{\gamma}=2j,\,j=0,\cdots,l_0\,\},$ then
\begin{equation*}
(1+A^{2\delta}(-\Delta_z))^{l_0}c(x+y,z+\xi,\eta)=\sum_{{\gamma}\in
P_{l_0}} C_{\gamma}
A^{\delta\abs{\gamma}}\partial_\xi^{\gamma}c(x+y,z+\xi,\eta),
\end{equation*}
and therefore
\begin{align}
|\partial_y^{\alpha-\beta}&\left((1+A^{2\delta}(-\Delta_z))^{l_0}c(x+y,z+\xi,\eta)\right)|\nonumber\\&\le
\sum_{\gamma\in P_{l_0}}C_{\gamma\alpha\beta}
A^{\delta\abs{\gamma}}(1+\abs{z+\xi}+\abs{\eta})^{m+\delta(\abs{\alpha}-\abs{\beta})-\rho\abs{\gamma}}.\label{esti2}
\end{align}

From \eqref{lappower}, \eqref{esti1} and \eqref{esti2}, we get
\begin{align}
&\abs{(-\Delta_y)^lq}\label{lappoweresti}\lesssim\\
&\left(1+A^{2\delta}\abs{y}^2\right)^{-l_0}\mathop{\sum_{\abs{\alpha}=2l}}_{\alpha_i
\text{ even}}\sum_{{\beta}\le \alpha} C_{\alpha\beta
l_0}\,A^{\delta\abs{\beta}}\sum_{\gamma\in
P_{l_0}}C_{\gamma\alpha\beta}
A^{\delta\abs{\gamma}}(1+\abs{z+\xi}+\abs{\eta})^{m+\delta(\abs{\alpha}-\abs{\beta})-\rho\abs{\gamma}}.\nonumber
\end{align}
Define the sets
\[
\Omega_1=\{z:\abs{z}\le {\textstyle \frac{A^{\delta}}{2}}\}, \quad
\Omega_2=\{z:{\textstyle\frac{A^{\delta}}{2}}\le \abs{z}\le
{\textstyle\frac{A}{2}}\},\quad \Omega_3=\{z:\abs{z}\ge
{\textstyle\frac{A}{2}}\}.
\]
We then have
\[
a(x,\xi,\eta)=\int_{\Omega_1}\int_y\cdots+\int_{\Omega_2}\int_y\cdots+\int_{\Omega_3}\int_y\cdots:=I_1+I_2+I_3.
\]
Note that
\begin{equation}\label{AOmega12}
\frac{1}{2} A\le 1+\abs{z+\xi}+\abs{\eta}\le \frac{3}{2} A,\qquad
z\in \Omega_1\cup\Omega_2.
\end{equation}
and
\begin{equation}\label{AOmega3}
 1+\abs{z+\xi}+\abs{\eta}\le A+\abs{z}\le 3\abs{z},\qquad z\in \Omega_3.
\end{equation}

{\bf Estimation for $I_1.$} The estimate \eqref{lappoweresti} with
$l=0,$ \eqref{AOmega12}, and $\delta-\rho\le0$ give, for $z\in
\Omega_1,$
\begin{align*}
\abs{q}&\le (1+A^{2\delta}\abs{y}^2)^{-l_0} \sum_{\gamma\in
P_{l_0}}{C_{\gamma}}
A^{\delta\abs{\gamma}}(1+\abs{z+\xi}+\abs{\eta})^{m-\rho\abs{\gamma}}\\&\le
(1+A^{2\delta}\abs{y}^2)^{-l_0}\sum_{\gamma\in P_{l_0}}C_{\gamma}
A^{m+(\delta-\rho)\abs{\gamma}}\\&\lesssim
(1+A^{2\delta}\abs{y}^2)^{-l_0}{A^m}.
\end{align*}
Therefore, since $2l_0>n,$
\begin{align*}
\abs{I_1}\lesssim
A^{m}\int_{\Omega_1}\int_y\frac{1}{(1+A^{2\delta}\abs{y}^2)^{l_0}}\,dydz\sim
A^{m}.
\end{align*}

{\bf Estimation for $I_2.$} Integration by parts gives
\begin{align*}
\int_y q(x,y,z,\xi,\eta)e^{-iz\cdot
y}\,dy&=\frac{1}{\abs{z}^{2l_0}}\int_y
q(x,y,z,\xi,\eta)(-\Delta_y)^{l_0}e^{-iz\cdot y}\,dy\\
&= \frac{1}{\abs{z}^{2l_0}}\int_y
(-\Delta_y)^{l_0}(q(x,y,z,\xi,\eta))e^{-iz\cdot y}\,dy.
\end{align*}
Using \eqref{lappoweresti} with $l=l_0,$  \eqref{AOmega12}, and
$\delta-\rho\le0$ we get, for $z\in \Omega_2,$
\begin{align*}
&\abs{(-\Delta_y)^{l_0}q}\\&\le
\left(1+A^{2\delta}\abs{y}^2\right)^{-l_0}
\mathop{\sum_{\abs{\alpha}=2{l_0}}}_{\alpha_i \text{
even}}\sum_{{\beta}\le \alpha} C_{\alpha\beta
l_0}\,A^{\delta\abs{\beta}} \sum_{\gamma\in
P_{l_0}}c_{\gamma\alpha\beta}
A^{\delta\abs{\gamma}}A^{m+\delta(\abs{\alpha}-\abs{\beta})-\rho\abs{\gamma}}\\
& \le \left(1+A^{2\delta}\abs{y}^2\right)^{-l_0}
\mathop{\sum_{\abs{\alpha}=2{l_0}}}_{\alpha_i \text{
even}}\sum_{{\beta}\le \alpha} C_{\alpha\beta l_0} \sum_{\gamma\in
P_{l_0}}c_{\gamma\alpha\beta}
A^{m+\delta\abs{\alpha}+(\delta-\rho)\abs{\gamma}}\\
&\lesssim
\frac{A^{m+2l_0\delta}}{\left(1+A^{2\delta}\abs{y}^2\right)^{l_0}}.
\end{align*}
Recalling that $2l_0>n,$ we get
\begin{align*}
\abs{I_2}&\le \int_{\Omega_2}\frac{1}{\abs{z}^{2l_0}}\int_y
\frac{A^{m+2l_0\delta}}{\left(1+A^{2\delta}\abs{y}^2\right)^{l_0}}\,dydz\\
&\lesssim A^{m+2l_0\delta-\delta n}\int_{\abs{z}\ge
\frac{A^\delta}{2}}\abs{z}^{-2l_0}\,dz\sim A^m.
\end{align*}

{\bf Estimation for $I_3.$} Let $l\in\bn$ to be chosen later. Again,
integration by parts gives
\begin{align*}
\int_y q(x,y,z,\xi,\eta)e^{-iz\cdot
y}\,dy&=\frac{1}{\abs{z}^{2l}}\int_y
q(x,y,z,\xi,\eta)(-\Delta_y)^{l}e^{-iz\cdot y}\,dy\\
&= \frac{1}{\abs{z}^{2l}}\int_y
(-\Delta_y)^{l}(q(x,y,z,\xi,\eta))e^{-iz\cdot y}\,dy.
\end{align*}
Using \eqref{lappoweresti} and \eqref{AOmega3}, and defining
$m_+=\max(0,m),$ we get, for $z\in \Omega_3,$
\begin{align*}
&\abs{(-\Delta_y)^lq}\\
&\lesssim
\left(1+A^{2\delta}\abs{y}^2\right)^{-l_0}\mathop{\sum_{\abs{\alpha}=2l}}_{\alpha_i
\text{ even}}\sum_{{\beta}\le \alpha} C_{\alpha\beta
l_0}\,A^{\delta\abs{\beta}} \sum_{\gamma\in
P_{l_0}}C_{\gamma\alpha\beta}
A^{\delta\abs{\gamma}}(1+\abs{z+\xi}+\abs{\eta})^{m+\delta(\abs{\alpha}-\abs{\beta})-\rho\abs{\gamma}}\\
&\lesssim
\left(1+A^{2\delta}\abs{y}^2\right)^{-l_0}\mathop{\sum_{\abs{\alpha}=2l}}_{\alpha_i
\text{ even}}\sum_{{\beta}\le \alpha} C_{\alpha\beta l_0}
\sum_{\gamma\in P_{l_0}}C_{\gamma,\alpha,\beta}
\abs{z}^{\delta(\abs{\beta}+\abs{\gamma})}\abs{z}^{m_++\delta(\abs{\alpha}-\abs{\beta})}\\
&\lesssim\left(1+A^{2\delta}\abs{y}^2\right)^{-l_0}
\abs{z}^{m_++\delta(2l+2l_0)}.
\end{align*}
We then have
\begin{align*}
\abs{I_3}&\lesssim
\int_{\Omega_3}\frac{1}{\abs{z}^{2l}}\int_y\left(1+A^{2\delta}\abs{y}^2\right)^{-l_0}
\abs{z}^{m_++\delta(2l+2l_0)}\,dydz\\
&\sim \int_{\abs{z}\ge
\frac{A}{2}}\abs{z}^{m_++2l_0\delta+2l(\delta-1)}\,dz\int_y\left(1+A^{2\delta}\abs{y}^2\right)^{-l_0}
\,dy\\
&\sim A^{-\delta n} \int_{\abs{z}\ge
\frac{A}{2}}\abs{z}^{m_++2l_0\delta+2l(\delta-1)}\,dz.
\end{align*}
We now choose $l\in\bn$ sufficiently large so that
\[m_++2l_0\delta+2l(\delta-1)<-n\quad \text{and}\quad -\delta
n+m_++2l_0\delta+2l(\delta-1)+n<m.\]
The existence of such an $l$ is guaranteed by the condition $0\le \delta<1$.

Finally,
\begin{align*}
\abs{I_3}&\lesssim A^{-\delta n+m_++2l_0\delta+2l(\delta-1)+n}\le
A^m.
\end{align*}
\end{proof}

\begin{proof}[Proof of Theorem~\ref{symbolic2}]
As in the proof of Theorem~\ref{symbolic} we use the representation
\eqref{arepresentation} for the symbol of $T^{*1}.$ Define
\[
a_\alpha(x,\xi,\eta):=\frac{i^{\abs{\alpha}}}{\alpha!}\partial_x^\alpha\partial_\xi^\alpha
c(x,\xi,\eta)=\frac{1}{\alpha!}\int\int\partial_\xi^\alpha
c(x+y,\xi,\eta)\,e^{-iz\cdot y}\,z^\alpha\,dydz.
\]
By the estimates \eqref{cesti}, $a_\al\in BS_{\rho,
\delta}^{m+(\delta-\rho)\abs{\al}}$ with constants independent of
the support of $\sigma$. We will show that
\begin{equation}\label{impob2}
\abs{\partial_x^{\al_1}\partial_\xi^{\al_2}\partial_\eta^{\al_3}(
a(x, \xi, \eta)-\sum_{|\al|<N} a_\al (x, \xi, \eta))}\leq C
(1+|\xi|+|\eta|)^{m+(\delta-\rho)N+\delta\abs{\al_1}-\rho(\abs{\al_2}+\abs{\al_3})},
\end{equation} where $C=C_{\al_1\al_2\al_3 N}$ is  independent of the
support of $\sigma.$ This then shows that
$a-\sum_{\abs{\al}<N}a_\al\in BS^{m+(\delta-\rho)N}_{\delta,\rho}$
and therefore we have \eqref{expansion1}. Now, by Taylor's theorem
\[
a(x, \xi, \eta)-\sum_{|\al|<N} a_\al (x, \xi,
\eta)=\sum_{\abs{\al}=N}\frac{1}{\al!}\int\int
\partial_{\xi}^\alpha c(x+y,\xi+tz,\eta)\,z^\alpha\,e^{-iz\cdot y}\,dydz,
\]
where $t\in (0,1)$ and $t=t(x,y,\xi,z,\eta).$ Note that, because of the estimates
\eqref{cesti}, it is enough to prove \eqref{impob2} for
$\al_1=\al_2=\al_3=0.$

Inequality \eqref{impob2}  follows  from  computations similar to
the ones in Theorem~\ref{symbolic}. We include them here for the
reader's convenience and for completeness.

Fix $N\in \bn_0$ and a multiindex $\al$ with $\abs{\al}=N.$ For
$l_0\in \bn,$ integration by parts gives
\begin{equation}\label{awithq2}
I_\al:=\int\int
\partial_{\xi}^\al c(x+y,\xi+tz,\eta)\,z^\al\,e^{-iz\cdot y}\,dydz=\int\int q(x,y,z,\xi,\eta)e^{-iz\cdot y}\,dydz,
\end{equation}
where
\[q(x,y,z,\xi,\eta)=\frac{(1+A^{2\delta}(-\Delta_z))^{l_0}(\partial_\xi^\al c(x+y,\xi+t z,\eta)\,z^\al)}{(1+A^{2\delta}\abs{y}^2)^{l_0}}.\]

We now estimate $(-\Delta_y)^lq$ for $l\in\bn.$
\begin{align}
&(-\Delta_y)^lq=\mathop{\sum_{\abs{\nu}=2l}}_{\nu_i
\text{ even}}C_\nu\,\partial_y^\nu q(x,y,z,\xi,\eta)\label{lappower2}\\
&= \mathop{\mathop{\sum_{\abs{\nu}=2l}}_{\nu_i \text{
even}}}_{{\beta}\le \nu}
C_{\nu\beta}\,\partial_y^{\beta}\left((1+A^{2\delta}\abs{y}^2)^{-l_0}\right)\partial_y^{\nu-\beta}\left(
(1+A^{2\delta}(-\Delta_z))^{l_0}(\partial_\xi^\al c(x+y,\xi+t
z,\eta)\,z^\al)\right).\nonumber
\end{align}

As before, if $P_{l_0}=\{\gamma=(\gamma_1,\cdots,\gamma_n): \gamma_i
\text{ even and } \abs{\gamma}=2j,\,j=1,\cdots,l_0\,\},$ then
\begin{align*}
&(1+A^{2\delta}(-\Delta_z))^{l_0}(\partial_\xi^\al c(x+y,\xi+t
z,\eta)\,z^\al)=\sum_{{\gamma}\in P_{l_0}} C_{\gamma}
A^{\delta\abs{\gamma}}\partial_z^{\gamma}(\partial_\xi^\al
c(x+y,\xi+t z,\eta)\,z^\al)\\
&=\mathop{\sum_{{\gamma}\in P_{l_0}}}_{\omega\le\gamma,
\omega\le\al} C_{\gamma\omega}
\,A^{\delta\abs{\gamma}}\,\partial_z^\omega
z^\al\,(\partial_\xi^{\al+\ga-\omega}
c)(x+y,\xi+tz,\eta)\,t^{\abs{\ga}-\abs{\omega}},
\end{align*}
and therefore
\begin{align}
|\partial_y^{\nu-\beta}&\left((1+A^{2\delta}(-\Delta_z))^{l_0}(\partial_\xi^\al
c(x+y,\xi+t z,\eta)\,z^\al)\right)|\nonumber\\&\le
\mathop{\sum_{\gamma\in P_{l_0}}}_{\omega\le\gamma,
\omega\le\al}C_{\ga\omega\nu\be}\,
A^{\delta\abs{\gamma}}\,\abs{z}^{N-\abs{\omega}}\,(1+\abs{\xi+tz}+\abs{\eta})^{m+\delta(\abs{\nu}-\abs{\beta})-\rho(N+\abs{\gamma}-\abs{\omega})}.\label{esti2_2}
\end{align}
From \eqref{lappower2}, \eqref{esti1} and \eqref{esti2_2}, we get
\begin{align}
&\abs{(-\Delta_y)^lq}\lesssim
\left(1+A^{2\delta}\abs{y}^2\right)^{-l_0}\label{lappoweresti2}\\
&\times\mathop{\mathop{\sum_{\abs{\nu}=2l,\nu_i \text{
even}}}_{\be\le\nu,\ga\in P_{l_0}}}_{\omega\le \ga,\omega\le \al}
C_{\ga\omega\nu\be l_0}\,
A^{\delta(\abs{\beta}+\abs{\gamma})}\,\,\abs{z}^{N-\abs{\omega}}\,(1+\abs{\xi+tz}+\abs{\eta})^{m+\delta(\abs{\nu}-\abs{\beta})-\rho(N+\abs{\gamma}-
\abs{\omega})}.\nonumber
\end{align}
Letting again
\[
\Omega_1=\{z:\abs{z}\le {\textstyle \frac{A^{\delta}}{2}}\}, \quad
\Omega_2=\{z:{\textstyle\frac{A^{\delta}}{2}}\le \abs{z}\le
{\textstyle\frac{A}{2}}\},\quad \Omega_3=\{z:\abs{z}\ge
{\textstyle\frac{A}{2}}\},
\]
we have
\[
I_\al=\int_{\Omega_1}\int_y\cdots+\int_{\Omega_2}\int_y\cdots+\int_{\Omega_3}\int_y\cdots:=
I_1+I_2+I_3.
\]
Note that
\begin{equation}\label{AOmega12_2}
\frac{1}{2} A\le 1+\abs{\xi+t z}+\abs{\eta}\le \frac{3}{2} A,\qquad
z\in \Omega_1\cup\Omega_2, \, t\in (0,1),
\end{equation}
and
\begin{equation}\label{AOmega3_2}
 1+\abs{\xi+t z}+\abs{\eta}\le A+\abs{z}\le 3\abs{z},\qquad z\in \Omega_3, \, t\in (0,1).
\end{equation}
The estimates in \eqref{impob2} for $\al_1=\al_2=\al_3=0$ follow if
we prove that
\[
\abs{I_i}\lesssim A^{m+{(\delta-\rho)}N},\qquad i=1,2,3.
\]

{\bf Estimation for $I_1.$} The estimate \eqref{lappoweresti2} with
$l=0,$ \eqref{AOmega12_2}, and $\delta-\rho< 0$ give, for $z\in
\Omega_1,$
\begin{align*}
\abs{q}&\le (1+A^{2\delta}\abs{y}^2)^{-l_0} \mathop{\sum_{\gamma\in
P_{l_0}}}_{\omega\le \gamma, \omega\le\al}{C_{\gamma\omega l_0}}\,
A^{m+(\delta-\rho)(\abs{\gamma}+N-\abs{\omega})}&\lesssim
\frac{A^{m+(\delta-\rho)N}}{(1+A^{2\delta}\abs{y}^2)^{l_0}}.
\end{align*}
Therefore, if we choose $2l_0>n,$ we get
\begin{align*}
\abs{I_1}\lesssim
A^{m+(\delta-\rho)N}\int_{\Omega_1}\int_y\frac{1}{(1+A^{2\delta}\abs{y}^2)^{l_0}}\,dydz\sim
A^{m+(\delta-\rho)N}.
\end{align*}

{\bf Estimation for $I_2.$} Integration by parts gives
\begin{align*}
\int_y q(x,y,z,\xi,\eta)e^{-iz\cdot
y}\,dy&=\frac{1}{\abs{z}^{2l_0}}\int_y
q(x,y,z,\xi,\eta)(-\Delta_y)^{l_0}e^{-iz\cdot y}\,dy\\
&= \frac{1}{\abs{z}^{2l_0}}\int_y
(-\Delta_y)^{l_0}(q(x,y,z,\xi,\eta))e^{-iz\cdot y}\,dy.
\end{align*}

Using \eqref{lappoweresti2} with $l=l_0,$  \eqref{AOmega12_2}, and
$\delta-\rho<0$ we get, for $z\in \Omega_2,$
\begin{align*}
&\abs{(-\Delta_y)^{l_0}q}\\&\le
\left(1+A^{2\delta}\abs{y}^2\right)^{-l_0}\mathop{\mathop{\sum_{\abs{\nu}=2l_0,\nu_i
\text{ even}}}_{\be\le\nu,\ga\in P_{l_0}}}_{\omega\le \ga,\omega\le
\al} C_{\ga\omega\nu\be l_0}\,
A^{\delta\abs{\gamma}}\,\abs{z}^{N-\abs{\omega}}\,A^{m+\delta2l_0-\rho(N+\abs{\gamma}-\abs{\omega})}.
\end{align*}
Choosing $l_0$ such that $2l_0>N+n,$
\begin{align*}
\abs{I_2}&\le A^{-n\delta}
\mathop{\mathop{\sum_{\abs{\nu}=2l_0,\nu_i \text{
even}}}_{\be\le\nu,\ga\in P_{l_0}}}_{\omega\le \ga,\omega\le \al}
C_{\ga\omega\nu\be l_0}\,
A^{\delta\abs{\gamma}}\,\,A^{m+\delta2l_0-\rho(N+\abs{\gamma}-\abs{\omega})}\int_{\abs{z}\ge
\frac{A^\delta}{2}}
\,\abs{z}^{N-\abs{\omega}-2l_0}dz\\
&\lesssim A^{m+(\delta-\rho)N}.
\end{align*}

{\bf Estimation for $I_3.$} Let $l\in\bn$ to be chosen later. Again,
integration by parts gives
\begin{align*}
\int_y q(x,y,z,\xi,\eta)e^{-iz\cdot
y}\,dy&=\frac{1}{\abs{z}^{2l}}\int_y
q(x,y,z,\xi,\eta)(-\Delta_y)^{l}e^{-iz\cdot y}\,dy\\
&= \frac{1}{\abs{z}^{2l}}\int_y
(-\Delta_y)^{l}(q(x,y,z,\xi,\eta))e^{-iz\cdot y}\,dy.
\end{align*}
Using \eqref{lappoweresti2} and \eqref{AOmega3_2}, and defining
$m_+=\max(0,m),$ we get, for $z\in \Omega_3,$
\begin{align*}
&\abs{(-\Delta_y)^lq} &\lesssim
\left(1+A^{2\delta}\abs{y}^2\right)^{-l_0}\mathop{\mathop{\sum_{\abs{\nu}=2l,\nu_i
\text{ even}}}_{\be\le\nu,\ga\in P_{l_0}}}_{\omega\le \ga,\omega\le
\al} C_{\ga\omega\nu\be l_0}\,
\abs{z}^{\delta(\abs{\beta}+2l_0)}\,\abs{z}^{N-\abs{\omega}}\,\abs{z}^{m_++\delta(2l-\abs{\beta})}
\end{align*}
We then have
\begin{align*}
\abs{I_3}&\lesssim A^{-\delta
n}\mathop{\mathop{\sum_{\abs{\nu}=2l,\nu_i \text{
even}}}_{\be\le\nu,\ga\in P_{l_0}}}_{\omega\le \ga,\omega\le \al}
C_{\ga\omega\nu\be l_0}\, \int_{\abs{z}\ge
\frac{A}{2}}\abs{z}^{m_++2l(\delta-1)+ 2l_0\delta+N}\,dz.
\end{align*}
Choosing $l\in\bn$ sufficiently large so that
\[m_++2l(\delta-1)+ 2l_0\delta+N<-n\quad \text{and}\quad -\delta n+
m_++2l(\delta-1)+ 2l_0\delta+N+n<m+(\delta-\rho)N,\] we obtain
\begin{align*}
\abs{I_3}&\lesssim A^{-\delta n+ m_++2l(\delta-1)+
2l_0\delta+N+n}\le A^{m+(\delta-\rho)N}.
\end{align*}
\end{proof}

\section{Proofs of Theorem \ref{symbolic3} and Theorem \ref{symbolic4}}\label{secc:sym34}

\begin{proof}[Proof of Theorem~\ref{symbolic3}]
The second part of the statement is immediate if we write for all
$N>0$,
$$a-b=(a-\sum_{j=0}^{N-1}a_j)-(b-\sum_{j=0}^{N-1}a_j)\in BS_{\rho, \delta}^{m_N},$$
and recall that $m_N\searrow -\infty$ as $N\rightarrow\infty$.

For the proof of the first part of Theorem \ref{symbolic3} we
proceed by explicitly constructing $a$. Let $\psi\in C_c^\infty
(\brn\times\brn)$ such that $0\leq \psi\leq 1$, $\psi (\xi, \eta)=0$
on $\{(\xi, \eta): |\xi|+|\eta|\leq 1\}$ and $\psi (\xi, \eta)=1$ on
$\{(\xi, \eta): |\xi|+|\eta|\geq 2\}$. Define
$$a(x, \xi, \eta)=\sum_{j=0}^\infty \psi (\epsilon_j\xi, \epsilon_j\eta)a_j (x, \xi, \eta),$$
where $\epsilon_j\searrow 0$ as $j\rightarrow\infty$ is an
appropriately chosen sequence of numbers in $(0, 1)$ so that $a\in
BS_{\rho, \delta}^{m_0}$. The choice of this sequence will be made
explicit below.

For each fixed $\epsilon\in (0,1)$ we have
\begin{enumerate}[(a)]

\item\label{(a)} $\psi (\epsilon\xi, \epsilon\eta)=0$ for $|\xi|+|\eta|\leq
1/\epsilon$;

\item\label{(b)} $\psi (\epsilon\xi, \epsilon\eta)=1$ for $|\xi|+|\eta|\geq
2/\epsilon$;

\item\label{(c)} For all $\be, \ga$ such that $|\be|+|\ga|\geq 1$,
$\partial_\xi^\be\partial_\eta^\ga \psi (\epsilon\xi,
\epsilon\eta)=0$ for $|\xi|+|\eta|\leq 1/\epsilon$ or
$|\xi|+|\eta|\geq 2/\epsilon$;

\item\label{(d)} $|\partial_\xi^\be\partial_\eta^\ga \psi (\epsilon\xi,
\epsilon\eta)|\leq c_{\be\ga}\epsilon^{|\be|+|\ga|}.$
\end{enumerate}

In particular, because of \eqref{(a)} and \eqref{(b)} (that is, we
only care about pairs $(\xi, \eta)$ such that
$1/\epsilon<|\xi|+|\eta|<2/\epsilon$), \eqref{(d)} is equivalent to

\begin{enumerate}[(e)]
\item\label{(e)} $|\partial_\xi^\be\partial_\eta^\ga \psi (\epsilon\xi, \epsilon\eta)|\leq c_{\be\ga} (1+|\xi|+|\eta|)^{-|\be|-|\ga|}.$
\end{enumerate}
This in turn is equivalent to saying that the family $\{\psi
(\epsilon\xi, \epsilon\eta)\}_{0<\epsilon<1}$ represents a bounded
set in $BS_{1, 0}^0$ (endowed with the topology induced by
appropriate semi-norms; see \cite{bt1}).

Based on the estimate (e) and $a_j\in BS_{\rho, \delta}^{m_j}$, we
can control each of the terms in the sum that defines $a$. By using
Leibniz' rule, we immediately obtain \bq\label{piece1}
|\partial_\xi^\be\partial_\eta^\ga \psi (\epsilon_j\xi,
\epsilon_j\eta)a_j (x, \xi, \eta)|\leq C_{j, \al\be\ga}
(1+|\xi|+|\eta|)^{m_j+\delta |\alpha|-\rho (|\be|+|\ga|)}. \eq Let
us now select $\epsilon_j$ such that $C_{j, \al\be\ga}\epsilon_j\leq
2^{-j}$ for all $|\al+\be+\ga|\leq j$. Due to (b) and
\eqref{piece1}, we can therefore write, for all $|\al+\be+\ga|\leq
j$, \bq\label{piece2} |\partial_\xi^\be\partial_\eta^\ga \psi
(\epsilon_j\xi, \epsilon_j\eta)a_j (x, \xi, \eta)|\leq 2^{-j}
(1+|\xi|+|\eta|)^{m_j+1+\delta |\alpha|-\rho (|\be|+|\ga|)}. \eq
Now, for a fixed $(x, \xi, \eta)$, the sum defining $a(x, \xi,
\eta)$ is finite. Indeed, by (a), if infinitely many terms
corresponding to a subsequence $(j_k)$ are non-zero, we necessarily
have $|\xi|+|\eta|>1/\epsilon_{j_k}\rightarrow \infty$ as
$k\rightarrow\infty$, a contradiction. In particular, we also have $a\in
C^\infty (\brn\times\brn\times\brn)$.

Fix then a triple of multi-indices $(\al, \be, \ga)$ and let $J>0$
be such that $|\al+\be+\ga|\leq J$. We split
$$a=S_1 (a)+S_2(a),$$
where
$$S_1 (a)=\sum_{j=0}^{J-1} \psi (\epsilon_j\xi, \epsilon_j\eta)a_j (x, \xi, \eta)$$
and
$$S_2 (a)=\sum_{j=J}^{\infty} \psi (\epsilon_j\xi, \epsilon_j\eta)a_j (x, \xi, \eta).$$

Since $S_1 (a)$ is a finite sum and each of its terms belongs to
$BS_{\rho, \delta}^{m_j}\subset BS_{\rho, \delta}^{m_0}$, we infer that
$S_1 (a)\in BS_{\rho, \delta}^{m_0}$.

To estimate $S_2 (a)$, recall first that for all $j\geq J$,
$m_j+1\leq m_J+1\leq m_0$, thus, by using \eqref{piece2}, we get
that
$$|\partial_x^\al\partial_\xi^\be\partial_\eta^\ga S_2 (a)|\leq \left(\sum_{j=J}^\infty 2^{-j}\right)
(1+|\xi|+|\eta|)^{m_0+\delta |\al|-\rho (|\be|+|\ga|)},$$ which
implies that $S_2 (a)\in BS_{\rho, \delta}^{m_0}$.

Thus, we conclude that $a\in BS_{\rho, \delta}^{m_0}$. We finally
arrive to the asymptotic expansion of $a$. We have
$$a-\sum_{j=0}^{N-1} a_j=\sum_{j=0}^{N-1}(\psi (\epsilon_j \xi, \epsilon_j\eta)-1)a_j + \sum_{j=N}^\infty \psi(\epsilon_j \xi, \epsilon_j\eta)
a_j.$$ In the first sum, because of (b), we only care about
$|\xi|+|\eta|<2/\epsilon_{N-1}$, and therefore we can achieve
whatever decay we wish. For the second sum, we proceed exactly as
above to show that it belongs to $BS_{\rho, \delta}^{m_N}$. The
proof is complete.
\end{proof}

\begin{proof}[Proof of Theorem~\ref{symbolic4}]
Note that, by Theorem \ref{symbolic3}, we know that there exists
some $b\in BS_{\rho, \delta}^{m_0}$ such that
$$b\sim \sum_{j=0}^\infty a_j.$$
Therefore, it will be sufficient to show that $a-b\in BS^{-\infty}$.
We start by noticing that
\begin{align}\begin{split}\nonumber
|a (x, \xi, \eta)&-b(x, \xi, \eta)|\\
&\leq |a(x,\xi,\eta)-\sum_{j=0}^{N-1}a_j(x, \xi, \eta)|+|b(x, \xi,
\eta)-\sum_{j=0}^{N-1}
a_j (x, \xi, \eta)|\\
&\leq C_N (1+|\xi|+|\eta|)^{-\mu_{N-1}}+\tilde C_N (1+|\xi|+|\eta|)^{m_N}\\
&\leq c_N (1+|\xi|+|\eta|)^{-N},
\end{split}
\end{align}
because both $-\mu_{N-1}$ and $m_N$ converge to $-\infty$ as
$N\rightarrow\infty$.

To estimate the derivatives of the difference $a-b$ we will employ
the following useful result; see the book by Taylor~\cite{tay}, p.41:

\medskip

\noindent  {\it If $K_1, K_2$ are two compact sets such that
$K_1\subset K_2^{\circ}\subset K_2$ and $u\in C_c^2 (\brn)$, then
$$\sum_{|\al|=1}\sup_{z\in K_1}|D^\al u(z)|\lesssim \sup_{z\in
K_2}|u(z)|\sum_{|\al|\leq 2}\sup_{z\in K_2}|D^\al u(z)|.$$} This is
an immediate consequence of the estimate of a first order partial
differential operator in terms of its second order partial
differential operator:
$$\|\partial u/\partial x_j\|_{L^\infty}^2\lesssim
\|u\|_{L^\infty}\|\partial^2 u/\partial x_j^2\|_{L^\infty},\qquad
u\in C_c^2 (\brn).$$

Let then $K$ be a compact set such that $x\in K$. Set $K_1=K\times
\{0\}\times \{0\}$, and let $K_2$ be a compact neighborhood of $K_1$. For
fixed $\xi, \eta$, define
$$F_{\xi, \eta}(x, \zeta, \zeta ')=a(x, \xi+\zeta, \eta+\zeta
')-b(x, \xi+\zeta, \eta+\zeta ').$$ We can write
\begin{align}\begin{split}\nonumber
&\sup_{x\in K} |\nabla_{x, \xi, \eta} (a-b)(x, \xi,
\eta)|^2=\sup_{(x, \zeta, \zeta ')\in K_1}|\nabla_{(x, \zeta, \zeta
')}
F_{\xi, \eta}(x, \zeta, \zeta ')|^2\\
&\lesssim \sup_{(x, \zeta, \zeta ')\in K_2}|F_{\xi, \eta}(x, \zeta,
\zeta ')|\sum_{|\al|\leq 2}
\sup_{(x, \zeta, \zeta ')\in K_2}|D^\al_{x,\zeta, \zeta '}(a-b)(x, \xi+\zeta, \eta+\zeta ')|\\
&\leq C_N\sup_{(x, \zeta, \zeta ')}(1+|\xi+\zeta|+|\eta+\zeta
'|)^{-N}(1+|\xi+\zeta|+|\eta+\zeta '|)^{\max (\mu, m_0+2(\delta-\rho
))}.
\end{split}
\end{align}
Since we have the freedom of choosing the compact neighborhood
$K_2$, we can assume that on it $|\zeta|\leq 1/3, |\zeta '|\leq
1/3.$ Then, by the triangle inequality, we have
$$1+|\xi+\zeta|+|\eta+\zeta '|\geq \frac{1}{3}(1+|\xi|+|\eta|),$$
and therefore we get, for all $N>0$, the estimate
$$|\partial_x\partial_\xi\partial_\eta (a-b)(x, \xi, \eta)|\leq C_N3^N2^{\max (\mu, m_0+2(\delta-\rho ))}(1+|\xi|+|\eta|)^{-N}.$$
Analogously, we will be able to control all the derivatives of $a-b$
by
$$|\partial_x^\al\partial_\xi^\be\partial_\eta^\ga (a-b)(x, \xi, \eta)|\leq C_N(\al, \be,\ga) (1+|\xi|+|\eta|)^{-N}.$$
This proves that $a-b\in BS^{-\infty}$ and the proof is complete.
\end{proof}

\section{Pointwise kernel estimates.}\label{secc:kernel}

In this section we will describe decay/blow-up estimates for
bilinear kernels associated to pseudodifferntial operators. These estimates can be summarized as
follows.

\begin{theorem}  Let $p \in BS^0_{\rho,\delta}, $ $0 < \rho \leq 1$, $0 \leq \delta < 1$, $m \in \re$, and let $k(x,y,z)$ denote the distributional kernel of associated bilinear pseudodifferential operator $T_p$. Let $\zp$ denote the set of non-negative integers and for $x, y, z \in \brn$, set
$$
S(x,y,z) = |x-y|+|x-z|+|y-z|.
$$
\begin{enumerate}[(i)]

\item Given $\alpha, \beta, \gamma \in \zpn$, there exists $N_0 \in \zp$ such that for each $N \geq N_0$,
$$
\sup_{(x,y,z): S(x,y,z) > 0} S(x,y,z)^N |D^\alpha_x D^\beta_y
D^\gamma_z k(x,y,z)| < \infty
$$
\item Suppose that $p$ has compact support in $(\xi, \eta)$ uniformly in $x$. Then $k$ is smooth, and given $\alpha, \beta, \gamma \in \zpn$ and
$N_0 \in \zp$, there exists $C > 0$ such that for all $x, y, z \in \brn$ with $S(x,y,z)>0$
    $$
    |D^\alpha_x D^\beta_y D^\gamma_z k(x,y,z)| \leq C (1 + S(x,y,z))^{-N}.
    $$
\item Suppose that $m + M + 2n < 0$ for some $M \in \zp$. Then $k$ is a bounded continuous function with bounded continuous derivatives of order $\leq M$.

\item Suppose that $m + M + 2n = 0$ for some $M \in \zp$. Then there exists a constant $C > 0$ such that for all $x, y, z \in \brn$ with $S(x,y,z)>0$,
$$
\sup_{|\alpha + \beta + \gamma|=M}|D^\alpha_x D^\beta_y D^\gamma_z
k(x,y,z)| \leq C |\log |S(x,y,z)||.
$$

\item Suppose that $m + M + 2n > 0$ for some $M \in \zp$. Then, given $\alpha, \beta, \gamma \in \zpn$, there exists a positive constant $C$ such that
for all $x, y, z \in \brn$ with $S(x,y,z) > 0$,
\begin{align*}
\sup\limits_{|\alpha + \beta + \gamma|=M} |\partial_x^\alpha
\partial_y^\beta \partial_z^\gamma k(x,y,z)| \leq C S(x,y,z)^{-(m +
M + 2n)/\rho}.
\end{align*}
\end{enumerate}
\end{theorem}

\begin{proof} Given a bilinear symbol $p(x,\xi, \eta)$ with $x, \xi, \eta \in \brn$, set $X=(x_1, x_2) \in \rtn$, $\zeta = (\xi, \eta) \in \rtn$ and define the linear symbol $P$ in $\rtn$ as
$$
P(X,\zeta) = p\left( \frac{x_1+x_2}{2}, \xi, \eta\right).
$$
It follows easily that if $p \in BS^{m}_{\rho, \delta}(\brn)$, for
some $m \in \re$ and $\rho, \delta \in [0,1]$, then $P \in
S^m_{\rho,\delta}(\rtn)$. Indeed, given $\alpha = (\alpha_1,
\alpha_2), \beta = (\beta_1, \beta_2) \in \zptn$
\begin{align*}
\left|\partial_X^\alpha \partial_\zeta^\beta P(X,\zeta) \right|& = \left|  \left(\frac{1}{2}\right)^{|\alpha_1 + \alpha_2|} (\partial_x^{\alpha_1+\alpha_2} \partial_\xi^{\beta_1} \partial_\eta^{\beta_2}p)\left(\frac{x_1+x_2}{2}, \xi, \eta\right)\right|\\
& \leq C_{\alpha, \beta} \left(\frac{1}{2}\right)^{|\alpha_1 + \alpha_2|}   (1+|\xi|+|\eta|)^{m - \rho |\beta_1 + \beta_2| + \delta |\alpha_1 + \alpha_2|}\\
& = C_{\alpha, \beta} \left(\frac{1}{2}\right)^{|\alpha_1 +
\alpha_2|}   (1+|\zeta|)^{m - \rho |\beta| + \delta |\alpha|}.
\end{align*}
In $\rtn$, for the associated linear operator $T_P$ we now have
\begin{align*}
T_P(F)(X) & = \int_{\rtn} P(X,\zeta) e^{i X \cdot \zeta} \widehat{F}(\zeta) d\zeta\\
 & = \int_{\brn} \int_{\brn} p\left( \frac{x_1+x_2}{2}, \xi, \eta\right) e^{i x_1 \cdot \xi} e^{i x_2 \cdot \eta} \widehat{F}(\xi, \eta) \, d\xi d\eta,
\end{align*}
and also
$$
T_P(F)(X) = \int_{\rtn} K(X,Y) F(Y) dY,
$$
where
$$
K(X,Y) = \mathcal{F}_{2n}(P(X, \cdot))(Y-X), \quad X, Y \in \rtn,
$$
and $\mathcal{F}_{2n}$ denotes the Fourier transform in $\rtn$.
Next, for the bilinear symbol $p$ we write
\begin{align*}
T_p(f,g)(x) & =\int_{\brn} \int_{\brn} p(x,\xi, \eta) e^{i x \cdot (\xi + \eta)} \hat{f}(\xi) \hat{g}(\eta) d\xi d\eta\\
&  = \int_{\brn} \int_{\brn} p(x,\xi, \eta) e^{i x \cdot \xi} e^{i x \cdot \eta} \widehat{f \otimes g}(\xi, \eta) d\xi d\eta\\
& = T_P(f \otimes g)(x,x)\\
& = \int_{\brn} \int_{\brn} K( (x,x), (y,z)) (f \otimes g) (y,z) dy dz\\
& = \int_{\brn} \int_{\brn} K( (x,x), (y,z)) f(y) g(z) dy dz.
\end{align*}
Therefore, the distributional bilinear kernel $k(x,y,z)$ of the
bilinear operator $T_p$ is given by
$$
k(x,y,z) = K( (x,x), (y,z)), \quad x, y, z \in \brn,
$$
where $K(X,Y)$ is the distributional linear kernel associated to the
linear operator $T_P$ in $\rtn$. Finally, the pointwise estimates
for linear kernels associated to symbols in $S^m_{\rho, \delta}(
\rtn)$ in \cite[Theorem 1.1]{AH90} imply the desired pointwise
estimates for the bilinear kernel $k(x,y,z)$.
\end{proof}

\section{An $L^2 \times W_0^{s,\infty} \rightarrow L^2$ boundedness property.}

If  $\sigma\in BS^0_{\rho,\delta},$  by freezing $g,$ $T_\sigma(\cdot,g)$ can be regarded as a
linear pseudodifferential  operator (with symbol depending on $g$),
that is,
\[T_\sigma(f,g)(x)=\int_\xi \sigma_g(x,\xi)\hat{f}(\xi)e^{i\xi x}\,d\xi,\]
where
\[\sigma_g(x,\xi)=\int_{\eta}\sigma(x,\xi,\eta)\hat{g}(\eta)e^{i\eta x}\,d\eta.\]
Moreover, the well-known $L^2$ boundedness of a linear
pseudodifferential operator asserts that if $\tau\in S^{0}_{\rho,
\delta},$ $0\le \delta\le\rho\le 1,$ $\delta<1,$ there exist
constants $C_0$ and $k\in\bn$ (independent of $\tau$) such that
\[\|T_\tau(u)\|_{L^2}\le C_0 \abs{\tau}_{k}\|u\|_{L^2},\quad u\in\mathcal{S}(\brn),\]
where
\[\abs{\tau}_{k}=\max_{\abs{\alpha}, \abs{\beta}\le k}\sup_{x,\,\xi}\abs{\partial_x^\alpha\partial_\xi^\beta \tau(x,\xi)}
(1+\abs{\xi})^{-\delta\abs{\alpha}+\rho\abs{\beta}}.\] In fact, $k$
can be taken equal to $[n/2]+1$, see \cite[p. 30]{cm1}.

\begin{theorem}\label{LWL}

Let $\sigma\in BS^0_{\rho,\delta},$ $0\le\delta\le \rho\le 1,$ $\delta<1.$ Then
$$
T_\sigma : L^2 \times W_0^{s, \infty} \rightarrow L^2,
$$
where $s$ is any integer satisfying
\begin{equation}\label{choices}
s > \frac{[n/2]+1}{1-\delta}+n.
\end{equation}
Moreover, if $g\in C^s_c(\brn)$ then $\sigma_g\in S^0_{\rho,\delta}$, and
\[
\abs{\sigma_g}_{[n/2]+1}\lesssim
\|g\|_{W_0^{s,\infty}}:=\sup_{\abs{\gamma}\le
s}\|D^{\gamma}g\|_{L^\infty}
\]
with a constant depending only on the  $BS^0_{\rho,\delta}$-norm of
$\sigma$ up to order $n+2$.
\end{theorem}

\begin{remark}
Note that Theorem \ref{LWL} includes the case $0 \leq \rho=\delta < 1$ and, in particular, the case $\rho = \delta = 0$, where, as pointed
out in the introduction, the mapping from $L^2(\brn)\times L^\infty(\brn)$ into $L^2(\brn)$ fails.
\end{remark}

\begin{proof} Let $g\in C^s_c(\brn)$ and $\sigma\in BS^0_{\rho,\delta},$
$0\le\delta\le\rho\le 1,$ $\delta<1.$ We assume $\sigma$ has
compact support so all calculations below can be justified. However,
all constants are independent of the support of $\sigma$ and a
limiting argument proves the result when $\sigma$ does not
have compact support (see \cite{bt1}).

Fix multiindices $\alpha$ and $\beta$ such that $\abs{\alpha},
\abs{\beta}\le [n/2]+1.$ Define $A=1+\abs{\xi}$ and let $l_0\in \bn$
(to be chosen later) and
$P_{l_0}=\{\gamma=(\gamma_1,\cdots,\gamma_n):\gamma_i \text{ is even
and }\abs{\gamma}=2j, j=0, \dots, l_0\}.$ We have,
\begin{align*}
& \partial_x^\alpha\partial_\xi^\beta
\sigma_g(x,\xi)=\sum_{\gamma\le
\alpha}c_{\gamma,\alpha}\int_z\int_y\partial_x^\gamma\partial_\xi^\beta\sigma(x,\xi,z)z^{\alpha-\gamma}e^{izy}g(x-y)\,dydz\\
&=\sum_{\gamma\le
\alpha}c_{\gamma,\alpha}\int_{z}\int_{y}(1+A^{2\delta}(-\Delta_z))^{l_0}(\partial_x^\gamma\partial_\xi^\beta\sigma(x,\xi,z)z^{\alpha-\gamma})
\frac{e^{izy}g(x-y)}{(1+A^{2\delta}\abs{y}^2)^{l_0}}\,dydz\\
&=\sum_{\gamma\le \alpha, \theta\in
P_{l_0}}c_{\gamma,\alpha,\theta}\int_{z}\int_{y}A^{\delta\abs{\theta}}\partial_z^\theta(\partial_x^\gamma\partial_\xi^\beta\sigma(x,\xi,z)z^{\alpha-\gamma})
\frac{e^{izy}g(x-y)}{(1+A^{2\delta}\abs{y}^2)^{l_0}}\,dydz\\
&=\mathop{\sum_{\gamma\le \alpha, \theta\in P_{l_0}}}_{\omega\le
\text{min}\{\theta, \alpha-\gamma\}}c_{\gamma,\alpha,\theta,
\omega}\int_{z}\int_{y}A^{\delta\abs{\theta}}\partial_x^\gamma\partial_\xi^\beta\partial_z^{\theta-\omega}\sigma(x,\xi,z)z^{\alpha-\gamma-\omega}
\frac{e^{izy}g(x-y)}{(1+A^{2\delta}\abs{y}^2)^{l_0}}\,dydz.
\end{align*}

Fix $\gamma\le \alpha,$ $\theta\in P_{l_0},$ $\omega\le
\text{min}\{\theta, \alpha-\gamma\}$ and set
\[p(x,\xi)=\int_{z}\int_{y}A^{\delta\abs{\theta}}\partial_x^\gamma\partial_\xi^\beta\partial_z^{\theta-\omega}\sigma(x,\xi,z)z^{\alpha-\gamma-\omega}
\frac{e^{izy}g(x-y)}{(1+A^{2\delta}\abs{y}^2)^{l_0}}\,dydz.\]

Define the sets
\[
\Omega_1=\{z:\abs{z}\le {\textstyle \frac{A^{\delta}}{2}}\}, \quad
\Omega_2=\{z:{\textstyle\frac{A^{\delta}}{2}}\le \abs{z}\le
{\textstyle\frac{A}{2}}\},\quad \Omega_3=\{z:\abs{z}\ge
{\textstyle\frac{A}{2}}\}.
\]
We then have
\[
p(x,\xi)=\int_{\Omega_1}\int_y\cdots+\int_{\Omega_2}\int_y\cdots+\int_{\Omega_3}\int_y\cdots:=I_1+I_2+I_3.
\]

Note that
\begin{equation*}
 A\le 1+\abs{\xi}+\abs{z}\le 2 A,\qquad z\in
\Omega_1\cup\Omega_2
\end{equation*}
and that
\begin{equation*}
 \abs{z}\le 1+\abs{\xi}+\abs{z}\le 2 \abs{z},\qquad z\in
\Omega_3
\end{equation*}

\bigskip

{\bf Estimation for $I_1.$} Choose $l_0$ such that $2l_0>n.$ Then
\begin{align*}
\abs{I_1}&\lesssim
A^{\delta\abs{\theta}}\|g\|_{L^\infty}\int_y\int_{\abs{z}\le
\frac{A^\delta}{2}}(1+\abs{\xi}+\abs{z})^{\delta\abs{\gamma}-
\rho(\abs{\beta}+\abs{\theta}-\abs{\omega})}\frac{\abs{z}^{\abs{\alpha}-\abs{\gamma}-\abs{\omega}}}{(1+A^{2\delta}\abs{y}^2)^{l_0}}\,dzdy\\
&\sim A^{\delta\abs{\theta}}\|g\|_{L^\infty} A^{\delta\abs{\gamma}-
\rho(\abs{\beta}+\abs{\theta}-\abs{\omega})} A^{-\delta n}
A^{\delta(\abs{\alpha}-\abs{\gamma}-\abs{\omega}+n)}\\
&= A^{(\delta-\rho)(\abs{\theta}-\abs{w})}
A^{\delta\abs{\alpha}-\rho\abs{\beta}}\|g\|_{L^\infty}\\
&\le A^{\delta\abs{\alpha}-\rho\abs{\beta}}\|g\|_{L^\infty}.
\end{align*}
Note that in the last inequality we have used that $\delta\le \rho$
and $\abs{\theta}-\abs{w}\ge 0.$

\bigskip

{\bf Estimation for $I_2.$} Let $l\in\bn$ to be chosen later. We have
\begin{align*}
I_2&= \int_{\Omega_2}\int_y
A^{\delta\abs{\theta}}\partial_x^\gamma\partial_\xi^\beta\partial_z^{\theta-\omega}\sigma(x,\xi,z)z^{\alpha-\gamma-\omega}
\frac{g(x-y)}{(1+A^{2\delta}\abs{y}^2)^{l_0}}
\frac{(-\Delta_y)^{l}(e^{izy})}{\abs{z}^{2l}}\,dydz\\
&= \int_{\Omega_2}A^{\delta\abs{\theta}}
\frac{z^{\alpha-\gamma-\omega}}{\abs{z}^{2l}}\partial_x^\gamma\partial_\xi^\beta\partial_z^{\theta-\omega}
\sigma(x,\xi,z)\int_y(-\Delta_y)^{l}\left(\frac{g(x-y)}{(1+A^{2\delta}\abs{y}^2)^{l_0}}\right)e^{izy}\,dy\,dz.
\end{align*}

Now,
\begin{align*}
\abs{(-\Delta_y)^{l}\left(\frac{g(x-y)}{(1+A^{2\delta}\abs{y}^2)^{l_0}}\right)}&=\abs{\mathop{\sum_{\abs{\mu}=2l}}_{\mu_i
even, \nu\le \mu} C_{\nu\mu}\partial_y^\nu((1+A^{2\delta}
\abs{y}^2)^{-l_0})\partial_y^{\mu-\nu}g(x-y)}\\
&\le \mathop{\sum_{\abs{\mu}=2l}}_{\mu_i even, \nu\le \mu} C_{\nu\mu
l_0} \|D^{\mu-\nu}g\|_{L^\infty} A^{\delta\abs{\nu}}
(1+A^{2\delta}\abs{y}^2)^{-l_0}.
\end{align*}

Therefore,
\begin{align*}
& \abs{I_2}\le \int_{\Omega_2}A^{\delta\abs{\theta}}
\frac{\abs{z^{\alpha-\gamma-\omega}}}{\abs{z}^{2l}}\abs{\partial_x^\gamma\partial_\xi^\beta\partial_z^{\theta-\omega}
\sigma(x,\xi,z)} \mathop{\sum_{\abs{\mu}=2l}}_{\mu_i even, \nu\le
\mu} C_{\nu\mu l_0} \|D^{\mu-\nu}g\|_{L^\infty}
A^{\delta\abs{\nu}}A^{-\delta
n}\,dz\\
&\le \mathop{\sum_{\abs{\mu}=2l}}_{\mu_i even, \nu\le \mu} C_{\nu\mu
l_0} \|D^{\mu-\nu}g\|_{L^\infty}
A^{\delta(\abs{\theta}+\abs{\nu}-n)} \int\limits_{\Omega_2}
\abs{z}^{\abs{\alpha}-\abs{\gamma}-\abs{\omega}-2l}
(A+\abs{z})^{\delta\abs{\gamma}-\rho(\abs{\beta}+\abs{\theta}-\abs{\omega})}\,dz\\
&\le A^{\delta(\abs{\theta}+2l-n)} \|g\|_{W_0^{2l,\infty}}
A^{\delta\abs{\gamma}-\rho(\abs{\beta}+\abs{\theta}-\abs{\omega})}
A^{\delta(\abs{\alpha}-\abs{\gamma}-\abs{\omega}-2l+n)},
\end{align*}
where we have used that $A+\abs{z}\sim A$ on $\Omega_2$
and $l$ has been chosen so that
\begin{equation}\label{l1}
\abs{\alpha}-\abs{\gamma}-\abs{\omega}-2l+n<0.
\end{equation}
Thus we obtain,
\begin{align*}
\abs{I_2}&\le C\, A^{\delta(\abs{\theta}+2l-n)} \|g\|_{W_0^{2l,\infty}}
A^{\delta\abs{\gamma}-\rho(\abs{\beta}+\abs{\theta}-\abs{\omega})}
A^{\delta(\abs{\alpha}-\abs{\gamma}-\abs{\omega}-2l+n)}\\
& = C\, A^{\delta\abs{\alpha}-\rho\abs{\beta}}
A^{(\delta-\rho)(\abs{\theta}-\abs{w})} \|g\|_{W_0^{2l,\infty}}\\
&\le C\, A^{\delta\abs{\alpha}-\rho\abs{\beta}}  \|g\|_{W_0^{2l,\infty}},
\end{align*}
since $(\delta-\rho)(\abs{\theta}-\abs{w})\le 0.$

\bigskip

{\bf Estimation for $I_3.$} Here we impose some extra conditions to
$l$ above. As in the estimation for $B_2$ we have
\begin{align*}
\abs{I_3} &\le \mathop{\sum_{\abs{\mu}=2l}}_{\mu_i even, \nu\le \mu}
C_{\nu\mu l_0} \|D^{\mu-\nu}g\|_{L^\infty}
A^{\delta(\abs{\theta}+\abs{\nu}-n)} \int_{\Omega_3}
\abs{z}^{\abs{\alpha}-\abs{\gamma}-\abs{\omega}-2l}
(A+\abs{z})^{\delta\abs{\gamma}-\rho(\abs{\beta}+\abs{\theta}-\abs{\omega})}\,dz.
\end{align*}

Using that $A+\abs{z}\sim \abs{z}$ in $\Omega_3$ we have
\begin{align*}
\abs{I_3} &\le C\, \|g\|_{W_0^{2l,\infty}}
A^{\delta(\abs{\theta}+2l-n)} \int_{\Omega_3}
\abs{z}^{\abs{\alpha}-\abs{\gamma}-\abs{\omega}-2l+\delta\abs{\gamma}-\rho(\abs{\beta}+\abs{\theta}-\abs{\omega})}\,dz.\end{align*}

Now, let us choose $l$ as in the estimation of $I_2$ and satisfying
\[
\abs{\alpha}-\abs{\gamma}-\abs{\omega}-2l+\delta\abs{\gamma}-\rho(\abs{\beta}+\abs{\theta}-\abs{\omega})+n<0.
\]
For example, any choice of $l$ such that
$2l>\abs{\alpha}+n$ satisfies the inequality above. Then,
\begin{align*}
\abs{I_3} &\le C\, A^{\delta(\abs{\theta}+2l-n)}\, A^{\abs{\alpha}-\abs{\gamma}-\abs{\omega}-2l+\delta\abs{\gamma}-\rho(\abs{\beta}+\abs{\theta}-\abs{\omega})+n}\, \|g\|_{W_0^{2l,\infty}}\\
&= C\,
A^{\delta\abs{\gamma}-\rho\abs{\beta}}\,A^{(\delta-1)(2l-n)+\abs{\alpha}-\abs{\gamma}}\,
A^{(\delta-\rho)\abs{\theta}}\,
A^{(\rho-1)\abs{\omega}}\,\|g\|_{W_0^{2l,\infty}}.
\end{align*}
Since $\delta<1$, we can choose $l$ sufficiently large so that
\begin{equation}\label{l2}
(\delta-1)(2l-n)+\abs{\alpha}-\abs{\gamma}\le 0
\end{equation}
and since $\delta-\rho\le 0,$  $\rho-1\le 0$ and $\abs{\gamma}\le
\abs{\alpha}$ we obtain
\begin{align*}
\abs{I_3} &\le  C\,
A^{\delta\abs{\alpha}-\rho\abs{\beta}}\,\|g\|_{W_0^{2l,\infty}}.
\end{align*}
Finally, by choosing $l=s/2$, with $s$ as in \eqref{choices}, we guarantee that both conditions
\eqref{l1} and \eqref{l2} are satisfied, and the proof is complete.
\end{proof}

\begin{corollary}\label{leibm} Let $m \geq 0$ and $\sigma\in BS^m_{\rho,\delta},$ $0\le\delta\le \rho\le 1,$ $\delta<1.$ Then, if $s$ is any integer satisfying \eqref{choices}, the following fractional Leibniz rule-type inequality holds true
\begin{equation}
\norm{T_\sigma(f,g)}{L^2} \leq C \left(\norm{f}{W^{m,2}} \norm{g}{W_0^{s,\infty}} + \norm{f}{W_0^{s,\infty}} \norm{g}{W^{m,2}}  \right), \quad f, g \in C_c^\infty(\rn).
\end{equation}

\end{corollary}

\begin{proof} Corollary \ref{leibm} follows from Theorem \ref{LWL} and composition with Bessel potentials of order $m$, along the lines of Theorem 2 in \cite{bnt}. We only need to notice that, if $\sigma \in BS^m_{\rho,\delta}$ and $\phi$ is a $C^\infty$ function on $\re$ such that $0 \leq \phi \leq 1$, $\supp(\phi) \subset [-2, 2]$ and $\phi(r) + \phi(1/r) = 1$ on $[0,\infty)$, then, the symbols $\sigma_1$ and $\sigma_2$ defined by
$$
\sigma_1(x,\xi, \eta)= \sigma(x,\xi, \eta) \phi\left(\frac{1+|\xi|^2}{1+|\eta|^2} \right) (1+|\xi|^2)^{-m/2}
$$
and
$$
\sigma_2(x,\xi, \eta)= \sigma(x,\xi, \eta) \phi\left(\frac{1+|\eta|^2}{1+|\xi|^2} \right) (1+|\eta|^2)^{-m/2}
$$
satisfy $\sigma_1, \sigma_2 \in BS^0_{\rho,\delta}$, and the corresponding operators $T_\sigma$, $T_{\sigma_1}$, and $T_{\sigma_2}$ are related through
$$
T_\sigma(f,g) = T_{\sigma_1}(J^m f, g) + T_{\sigma_2}(f,J^m g),
$$
where $J^m$ denotes the linear Fourier multiplier with symbol $(1+|\xi|^2)^{m/2}$.
\end{proof}

\end{document}